\documentclass[11pt,a4paper]{amsart}
\usepackage{amssymb}
\usepackage{amsmath}
\usepackage{amsthm}
\usepackage{amscd,bm,bbm}
\usepackage{enumitem,graphicx,tikz, pst-node, pst-plot, pstricks}
\usepackage[a4paper]{geometry}
\usepackage{hyperref}
\usepackage{color}
\usepackage{cite}
\usetikzlibrary{patterns,shapes,decorations.pathreplacing}
\numberwithin{equation}{section}
\usepackage[foot]{amsaddr}
\usepackage[english]{babel}

\theoremstyle{plain}
\newtheorem{theorem}{Theorem}[section]
\newtheorem{proposition}[theorem]{Proposition}
\newtheorem{corollary}[theorem]{Corollary}
\newtheorem{lemma}[theorem]{Lemma}

\theoremstyle{remark}
\newtheorem{remark}[theorem]{Remark}

\newtheorem{example}[theorem]{Example}

\DeclareMathOperator*{\esssup}{ess\,sup}

\newcommand{\loc}{\operatorname{loc}}

\newcommand{\R}{\mathbb{R}}

\DeclareMathOperator*{\lip}{Lip}

\begin{document}

\title[Variability and the existence of rough integrals]{Variability and the existence of rough integrals with irregular coefficients}
\author[M. Hinz]{Michael Hinz}
\address[MH]{Universit\"at Bielefeld\\
Fakult\"at f\"ur Mathematik\\
Postfach 100131\\
33501 Bielefeld\\
Germany}
\email{mhinz@math.uni-bielefeld.de}

\author[J. M. T\"olle]{Jonas M. T\"olle}
\address[JMT]{Aalto University\\
Department of Mathematics and Systems Analysis\\
PO Box 11100 (Otakaari 1, Espoo)\\
00076 Aalto\\
Finland}
\email{jonas.tolle@aalto.fi}

\author[L. Viitasaari]{Lauri Viitasaari}
\address[LH]{Aalto University School of Business\\
Department of Information and Service Management\\
PO Box 21210 (Ekonominaukio 1, Espoo)\\
00076 Aalto\\
Finland}
\email{lauri.viitasaari@aalto.fi}

\keywords{Fractional rough path integrals;
variability; irregular multiplicative functionals;
functions of bounded variation;
Gaussian processes;
fractional Brownian motion}
\subjclass{Primary: 26B30; 46E35; 60G15; 60G17; 60G22; 60L20. Secondary: 26A33; 31B15; 42B20.}

\date{\today}
\thanks{JMT gratefully acknowledges partial support by the Magnus Ehrnrooth Foundation.\\
This work is licensed under the Creative Commons Attribution 4.0 International License. To view a copy of this license, visit \url{http://creativecommons.org/licenses/by/4.0/} or send a letter to Creative Commons, PO Box 1866, Mountain View, CA 94042, USA. \includegraphics[height=1em]{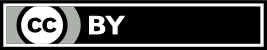}
Original work published in the Electronic Communications in Probability \textbf{30} (2025), no.~1, 1--12., \url{https://doi.org/10.1214/25-ECP656}.}
\begin{abstract}
Within the context of rough path analysis via fractional calculus, we show how variability can be used to prove the existence of integrals with respect to H\"older continuous multiplicative functionals in the case of Lipschitz coefficients with first order partial derivatives of bounded variation. We discuss applications to certain Gaussian processes, in particular, fractional Brownian motions with Hurst index $\frac13<H\leq \frac12$.
\end{abstract}
\maketitle

{\footnotesize
\tableofcontents
}
\section{Introduction}

It is well-known that under certain complementary regularity conditions on the integrand $X$ and the integrator $Y$, formulated in terms of H\"older or Sobolev regularity, or in terms of $p$-variation, the Stieltjes type integral $\int XdY$ of a path $X$ with respect to a path $Y$ exists and admits a typical estimate, see \cite{Lyons1994, Young, Zahle98, Zahle01}. These regularity conditions are sometimes referred to by saying that the pair $(X,Y)$ is in the so-called \emph{Young regime}. For a coefficient function $\varphi$ of sufficient H\"older regularity, even the pair $(\varphi(X),Y)$ is in the Young regime and, as a consequence, the integral $\int\varphi(X)\,dY$ exists, see Corollary \ref{C:compositioncase} below. By this fact, one can study differential equations driven by paths of fractional Brownian motion with Hurst index $\frac12<H<1$, see \cite{NualartRascanu, Zahle01}.

In \cite{HTV2020} and \cite{HTV2022}, we studied the situation where $\varphi$ lacks this sufficient regularity. We considered coefficient functions $\varphi$ with components $\varphi_j$ of bounded variation and proved that under a mutual diffusivity condition on the gradient measures $D\varphi_j$ and the occupation measure of $X$ the compositions $\varphi_j(X)$ are well-defined and sufficiently regular to ensure that $(\varphi(X),Y)$ is in the Young regime and $\int\varphi(X)\,dY$ exists. We call this condition the \emph{variability} of $X$ with respect to $\varphi$. See Corollary \ref{C:sobolev-membership} below.

Rough path analysis, introduced in \cite{Lyons1994, Lyons98, LyonsQian}, and, in a different formulation, in \cite{Gubinelli}, provides a vast toolkit for pathwise integrals beyond the Young regime. See \cite{LyonsQian, FH:20, FrizVictoir, LCL:07} for background. An approach to rough path analysis via fractional calculus, \cite{SKM}, was proposed and investigated in \cite{HN09}. Comprehensive further studies in this direction were provided in \cite{Ito2015, Ito2017, Ito2017b}. The results in \cite{HN09} suffice to study differential equations driven by paths of fractional Brownian motion with Hurst index $\frac13<H\leq \frac12$. The hypotheses of \cite{HN09} ensure the existence of $\int\varphi(X)\,dY$ for $\varphi$ with Lipschitz components having partial derivatives of sufficiently high H\"older regularity. We recall this in Theorem \ref{T:HN} below.

Here we show that in the setting of \cite{HN09}, a pragmatic variant of variability ensures the existence of $\int\varphi(X)\,dY$ for $\varphi$ with Lipschitz components having first order partial derivatives of bounded variation. This is proved in Theorem \ref{thm:existence-of-integral-H} below, which is a novel observation: Existing results in rough path analysis assume at least some H\"older regularity of first order partials, while in Theorem \ref{thm:existence-of-integral-H} these partials may be discontinuous. For integrals of fractional Brownian motion $B^H$ with Hurst index $\frac13<H\leq \frac12$ the result applies in two ways. In dimension one we can find an event $\Omega_1$ of full probability such that for any Lipschitz $\varphi$ with $BV$ derivative $\varphi'$ the integral $\int\varphi(B^H(\omega))dB^H(\omega)$ exists for any $\omega\in \Omega_1$. In dimension $m>1$ we can, for any finite nonnegative Borel measure $\mu$ on $\mathbb{R}^m$, find an event $\Omega_\mu$ of full probability on which the integral exists for any Lipschitz $\varphi$ with partials in $BV$ and having gradient measures dominated by $\mu$. This is weaker, but still for each $\omega\in \Omega_\mu$ the integral $\int\varphi(B^H(\omega))dB^H(\omega)$ converges in the deterministic sense. Note that approaches where also the convergence of integrals is stochastic give much stronger results, a low H\"older regularity of $\varphi$ suffices, see \cite[Proposition 3.5]{MP24} and \cite[Section 4]{MM23}. For $\frac12<H<1$ even bounded $\varphi$ is enough, \cite[Proposition 3.3]{MP24}.

In Section \ref{S:fraccalc}, we review known existence results for pathwise integrals in the Young regime and in the rough setting of \cite{HN09}. In Section \ref{S:varcomp}, we discuss our former results on compositions of paths and functions of bounded variation and on the existence of pathwise integrals in the Young regime. We prove a new result on the existence of pathwise integrals involving Lipschitz coefficients with partial derivatives of bounded variation under a weighted variability condition, Theorem \ref{thm:existence-of-integral-H}. In Section \ref{S:Gauss}, we discuss straightforward sufficient conditions 
ensuring the a.s. validity of the weighted variability condition for paths of certain Gaussian processes, including fractional Brownian motion.

\section{Pathwise integrals via fractional calculus}\label{S:fraccalc}

\subsection{Function spaces and fractional derivatives} 
Let $U\subset \mathbb{R}^n$ be open or the closure of an open set and $0<\beta<1$. Given a Borel function $f:U\to \mathbb{R}^m$ we write
\[[[f]]_{\beta}:=\sup_{x,y\in U,\ s\neq t}\frac{|f(x)-f(y)|}{|x-y|^{\beta}}\]
for its \emph{$\beta$-H\"older seminorm}. If $1\leq p<\infty$, we write 
\begin{equation}\label{E:Gagliardo}
[f]_{\beta,p}:= \left(\int_U \int_U \frac{|f(x)-f(y)|^p}{|x-y|^{n+\beta p}}\:dx\:dy\right)^{\frac{1}{p}}
\end{equation}
for its \emph{$(\beta,p)$-Gagliardo seminorm}, complemented by 
\[[f]_{\beta,\infty}:=\esssup_{x\in U}\int_U\frac{|f(x)-f(y)|}{|x-y|^{n+\beta}}\:dy\]
for the case $p=\infty$. As usual, we write $C^\beta(U;\mathbb{R}^m)$ for the space of $f$ such that 
$\|f\|_{C^\beta(U;\mathbb{R}^m)}:=\|f\|_{\sup}+[[f]]_{\beta}$ is finite and $W^{\beta,p}(U;\mathbb{R}^m)$, $1\leq p\leq \infty$, for the space
of $f$ such that $\Vert f \Vert_{W^{\beta,p}(U;\mathbb{R}^m)}: = \Vert f \Vert_{L^p(U;\mathbb{R}^m)} + [f]_{\beta,p}$ is finite. For bounded and open $U$ and $\beta'>\beta$ the space $C^{\beta'}(\overline{U};\mathbb{R}^m)$ is continuously embedded into $W^{\beta,p}(U;\mathbb{R}^m)$, where $\overline{U}$ denotes the closure of $U$.

Let $-\infty<a<b<\infty$, let $0<\alpha<1$ and let $f:[a,b]\to \mathbb{R}^m$ be a Borel function.  If $f$ is sufficiently regular, then the (left and right sided) Weyl-Marchaud fractional derivatives of $f$ of order $0<\alpha<1$, \cite{SKM}, are defined by
\[D_{a+}^\alpha f(t):=\frac{1}{\Gamma(1-\alpha)}\left(\frac{f(t)}{(t-a)^\alpha}+\alpha\int_a^t \frac{f(t)-f(s)}{(t-s)^{\alpha+1}}\,ds\right),\quad a\le t\le b,\]
and
\[D_{b-}^\alpha f(t):=\frac{(-1)^\alpha}{\Gamma(1-\alpha)}\left(\frac{f(t)}{(b-t)^\alpha}+\alpha\int_t^b \frac{f(t)-f(s)}{(s-t)^{\alpha+1}}\,ds\right),\quad a\le t\le b,\]
seen as elements of $L^1([a,b];\mathbb{R}^m)$. 

If $1\leq p\leq \infty$, $\beta>\alpha$ and $f\in W^{\beta,p}([a,b];\mathbb{R}^m)$ is continuous at $a$ respectively $b$, then $D_{a+}^\alpha (f-f(a))$ respectively $D_{b-}^\alpha (f-f(b))$ is defined and in $L^p([a,b];\mathbb{R}^m)$ with norm bounded by a universal constant times $\|f-f(a)\|_{W^{\beta,p}([a,b];\mathbb{R}^m)}$, respectively $\|f-f(b)\|_{W^{\beta,p}([a,b];\mathbb{R}^m)}$.

\subsection{Pathwise integrals via fractional calculus}

Let $-\infty<a<b<\infty$, let $X,Y:[a,b]\to \mathbb{R}^m$ be Borel functions, $X$ continuous at $a$ and $Y$ continuous at $a$ and $b$. Suppose that $1\leq u,v\leq \infty$ are such that $\frac{1}{u}+\frac{1}{v}\leq 1$ with the convention $\frac{1}{\infty}:=0$, $0<\alpha<1$, $D_{a+}^\alpha (X-X(a))\in L^u([a,b];\mathbb{R}^m)$ and $D_{b-}^\alpha (Y-Y(b))\in L^v([a,b];\mathbb{R}^m)$.
We write
\begin{equation}\label{E:basicint}
\int_a^b X\,dY:=(-1)^\alpha \int_a^b D_{a+}^\alpha (X(t)-X(a)) D_{b-}^{1-\alpha} (Y(t)-Y(b))\,dt+X(a)(Y(b)-Y(a)).
\end{equation}
This definition was introduced in \cite{Zahle98}, where it was also shown to be independent of the particular choice of $\alpha$, \cite[Proposition 2.1]{Zahle98}. To quote a variant of this well-known result, \cite{NualartRascanu, Zahle98, Zahle01}, we will say that $\int_a^b X\,dY$ in \eqref{E:basicint} \emph{exists} if for some range of $\alpha$ the integrands on the right-hand side in \eqref{E:basicint} are well-defined, the integral exists, and its value does not depend on the particular choice of $\alpha$ within that range. A proof of this particular variant can be found in \cite[Theorem 6.2]{HTV2022}.

\begin{proposition}
\label{prop:basic-GSL}
Let $0<\gamma,\delta<1$ and $1\leq p,q\leq \infty$ be such that 
\begin{equation}\label{E:gammatheta}
\gamma+\delta>1\quad\text{and}\quad \frac{1}{p}+\frac{1}{q}<\gamma+\delta.
\end{equation} 
Suppose that $X,Y:[a,b]\to \mathbb{R}^m$ are Borel functions, $X$ is continuous at $a$ and $Y$ is continuous at $a$ and $b$. If
$X \in W^{\gamma,p}([a,b];\mathbb{R}^m)$ and $Y\in W^{\delta,q}([a,b];\mathbb{R}^m)$, then $\int_a^b X\,dY$ in \eqref{E:basicint} exists and satisfies
\begin{equation}\label{E:basicest}
\Big|\int_a^b X\,dY-X(a)(Y(b)-Y(a))\Big|\leq c\:\|X-X(a)\|_{W^{\gamma,p}([a,b];\mathbb{R}^m)}\|Y-Y(b)\|_{W^{\delta,q}([a,b];\mathbb{R}^m)}
\end{equation}
where $c>0$ is a constant depending only on $a$, $b$, $\gamma$, $\delta$, $p$ and $q$.
\end{proposition}

\begin{remark}
\begin{enumerate}
\item[(i)] It is well-known that if $X\in C^\gamma([a,b];\mathbb{R}^m)$ and $Y\in C^\delta([a,b];\mathbb{R}^m)$ with $\gamma+\delta>1$, then $\int_a^b X\,dY$ exists and agrees with the Young, Lebesgue-Stieltjes and Riemann-Stieltjes integral of $X$ with respect to $Y$. In this case one can choose any $1-\delta<\alpha<\gamma$ in the right-hand side of \eqref{E:basicint}.  Details may be found in \cite{Gubinelli, Lyons1994, Young, Zahle98}.
\item[(ii)] If in the situation of Proposition \ref{prop:basic-GSL} we have $\gamma p<1$, then 
\[\int_a^b X\,dY=(-1)^\alpha \int_a^b D_{a+}^\alpha X(t) D_{b-}^{1-\alpha} (Y-Y(b))(t)\,dt,\]
which is true in particular for the situation in (i), see \cite[Remark on p. 340]{Zahle98}. 
\end{enumerate}
\end{remark}

\begin{corollary}\label{C:compositioncase}
Let $0<\beta,\delta<1$ and $1\leq p,q\leq \infty$. Suppose that $X:[a,b]\to \mathbb{R}^m$ and $Y:[a,b]\to \mathbb{R}^d$
are Borel functions, $Y$ is continuous at $b$. If $X \in W^{\beta,p}([a,b];\mathbb{R}^m)$, $Y\in W^{\delta,q}([a,b];\mathbb{R}^m)$, $0<\lambda<1$, $\varphi\in C^\lambda(\mathbb{R}^m;\mathbb{R}^d)$, \eqref{E:gammatheta} is satisfied with $\gamma$ 
replaced by $\lambda\beta$ and we have $\lambda\beta p<1$, then $\int_a^b\varphi(X)\,dY$ in \eqref{E:basicint} exists and satisfies 
\[\Big|\int_a^b \varphi(X)\,dY\Big|\leq c\:\Big\lbrace\|\varphi(X)\|_{L^p([a,b];\mathbb{R}^d)}+[[\varphi]]_\lambda [X]_{\beta,p}\Big\rbrace\|Y-Y(b)\|_{W^{\delta,q}([a,b];\mathbb{R}^m)},\]
where $c>0$ is a constant depending only on $a$, $b$, $\beta$, $\delta$, $p$, $q$ and $\lambda$.
\end{corollary}

\begin{proof}
Assume first that $\frac{1}{p}+\frac{1}{q}\leq 1$ and choose $\alpha\in (1-\delta,\lambda\beta)$.
The $L^p$-norm of the first summand in $D_{a+}^\alpha \varphi(X)$ is bounded by a constant times $\|\varphi(X)\|_{W^{\alpha,p}([a,b];\mathbb{R}^d)}$, see \cite[Lemma 4.33]{HTV2020}. This norm, and also the $L^p$-norm of the second summand, does not exceed a constant times $\|\varphi(X)\|_{W^{\lambda\beta,p}([a,b];\mathbb{R}^d)}$. Using the H\"older property of $\varphi$ on the Gagliardo seminorm part we find the first factor in the stated estimate. The case $1< \frac{1}{p}+\frac{1}{q}<\lambda\beta+\delta$ can be reduced to the first case, see the proof of \cite[Theorem 6.2]{HTV2022}.
\end{proof}

\subsection{Rough path analysis via fractional calculus} Approaches to pathwise integrals via fractional calculus beyond the Young regime were studied in \cite{HN09, BesaluNualart2011, Ito2015, Ito2017, Ito2017b}. We follow \cite[Section 3]{HN09}.

Fix $\frac{1}{3}<\beta<\frac{1}{2}$. Given a function $X:[a,b]\to \mathbb{R}^m$ we use the notation 
\[X_{s,t}=X(t)-X(s),\quad a\leq s<t\leq b.\]
Set $\Delta:=\{(s,t)\subset [a,b]\times [a,b]\;\colon\;a\le s\le t\le b\}$. A triple $(X,Y,X\otimes Y)$ is called an \emph{$(m,d)$-dimensional $\beta$-H\"older continuous multiplicative functional} on $[a,b]$ if $X=(X^1,...,X^m):[a,b]\to\R^m$ and $Y=(Y^1,...,Y^d):[a,b]\to\R^d$ are $\beta$-H\"older continuous functions and $X\otimes Y:\Delta\to\R^m\otimes\R^d$ is a continuous function such that 
\begin{equation}\label{E:chen}
(X\otimes Y)_{s,u}+(X\otimes Y)_{u,t}+X_{s,u}\otimes Y_{u,t}=(X\otimes Y)_{s,t},\quad a\leq s\le u\le t\leq b,
\end{equation}
and 
\begin{equation}\label{E:twoparam}
\left|(X\otimes Y)_{s,t}\right|\leq c\: |t-s|^{2\beta},\quad (s,t)\in\Delta
\end{equation}
with a constant $c>0$ depending only on $a$, $b$, $\beta$, $X$ and $Y$. See for instance \cite{Lyons98} or \cite[Definition 3.1]{HN09}. The space of all $(m,d)$-dimensional $\beta$-H\"older continuous multiplicative functionals on $[a,b]$ is denoted by $M_{m,d}^\beta([a,b])$.

\begin{remark}
\begin{enumerate}
\item[(i)] We will only use the $\beta$-H\"older continuity of $X$ and $Y$ and condition \eqref{E:twoparam}, Chen's relation \eqref{E:chen} will not be used.
\item[(ii)] Here the symbol $X\otimes Y$ is only a shortcut notation for an abstract two-parameter function. Given $X$ and $Y$, the validity of \eqref{E:chen} does not specify a unique tensor product of $X$ and $Y$. At the level of $X$ or $Y$, formula \eqref{E:chen} is invariant under perturbation by constant paths, whereas at the level of $X\otimes Y$ it is invariant under perturbation by increments of paths, see \cite{FH:20,LCL:07,Gubinelli,LyonsQian}. We refer to \cite{CG:14} for a general discussion of the second order object $X\otimes Y$ in the context of rough sheets.
\end{enumerate}
\end{remark}

\begin{example}\label{Ex:MF}
\begin{enumerate}
\item[(i)] It is well-known that if $X$ and $Y$ are continuously differentiable, then
$(X\otimes Y)_{s,t}^{i,j}=\int_{s<\xi<\eta<t} dX^i(\xi)\: dY^j(\eta)$, $i=1,\ldots,m$, $j=1,\ldots,d$,
gives a multiplicative functional $(X,Y,X\otimes Y)\in M_{m,d}^\beta([a,b])$, see \cite[p. 2693]{HN09} or \cite[Examples 7.1]{FrizVictoir}.
\item[(ii)] If $X=Y=B$ is a $d$-dimensional Brownian motion and $\frac13<\beta<\frac12$, then one can choose a version of the Stratonovich integral $(B\otimes B)_{s,t}=\int_s^t B_{r,s}d^\circ B_r$ so that $\mathbb{P}$-a.a. realizations of $(B,B,B\otimes B)$ are in $M_{d,d}^\beta([a,b])$, see \cite[Section 5]{HN09}.
\item[(iii)] If $X=Y=B^H$ is a $d$-dimensional fractional Brownian motion with Hurst index $\frac13<H<\frac12$, then one can use dyadic approximations or convolutions to define $B^H\otimes B^H$ as an $\mathbb{P}$-a.s. limit, see \cite[Theorem 2]{CQ2002} or \cite[Theorem 15.35]{FrizVictoir}. A different definition for $B^H\otimes B^H$ using kernel representations was given in \cite[Section 3.1]{NT:11}. For this definition \cite[Propositions 3.2 and 3.3]{NT:11} ensure that for $\frac13<\beta<H$ $\mathbb{P}$-a.a. realizations of $(B^H,B^H,B^H\otimes B^H)$ are in $M_{d,d}^\beta([a,b])$. A comparison of these definitions is provided in \cite[Section 5.1]{NT:11}.
\end{enumerate}
\end{example}

For $(X,Y,X\otimes Y)\in M_{m,d}^\beta([a,b])$ and $0<\gamma<1\wedge 2\beta$ we define the function
\[D_{b-}^{\gamma}(X\otimes Y)(r):=\frac{(-1)^{\gamma}}{\Gamma(1-\gamma)}\left(\frac{(X\otimes Y)_{r,b}}{(b-r)^{\gamma}}+\gamma\int_r^b \frac{(X\otimes Y)_{r,s}}{(s-r)^{\gamma+1}}\,ds\right),\quad a\le r<b.\]

Given a $\beta$-H\"older continuous function $X=(X^1,...,X^m):[a,b]\to\R^m$, a differentiable function $\varphi=(\varphi_1,...,\varphi_d):\mathbb{R}^m\to\mathbb{R}^d$ with $\lambda$-H\"older continuous first order partial derivatives $\partial_i \varphi=(\partial_i\varphi_1,...,\partial_i\varphi_d)$, $i=1,...,m$, and $0<\gamma<1\wedge \beta(1+\lambda)$,
the \emph{compensated fractional derivative} $\hat{D}^\gamma_{a+} \varphi(X)$ of $\varphi(X)$ is defined by
\begin{multline}
\hat{D}^\gamma_{a+} \varphi(X)(r)=\frac{1}{\Gamma(1-\alpha)}\Bigg(\frac{\varphi(X)(r)}{(r-a)^\gamma}\notag\\
+\gamma\int_a^r\frac{(\varphi(X))_{\theta,r}-\sum_{i=1}^m \partial_i \varphi(X)(\theta) X^i_{\theta,r}}{(r-\theta)^{\gamma+1}}\,d\theta\Bigg),\quad a\leq r\leq b.\notag
\end{multline}

We recall \cite[Definition 3.2]{HN09}: Given $(X,Y,X\otimes Y)\in M_{m,d}^\beta([a,b])$, a function $\varphi:\mathbb{R}^m\to\mathbb{R}^d$ and $0<\alpha<1$, we write 
\begin{align}
\int_a^b \varphi(X)\,dY:=&(-1)^\alpha \sum_{j=1}^d\int_a^b \hat{D}^\alpha_{a+}\varphi_j(X)(r) D_{b-}^{1-\alpha} (Y-Y(b))^j(r)\,dr\notag\\
&-(-1)^{2\alpha-1}\sum_{i=1}^m\sum_{j=1}^d \int_a^b D_{a+}^{2\alpha-1}(\partial_i \varphi_j(X))(r) D_{b-}^{2-2\alpha}(X\otimes Y)^{i,j}_{r,b}\,dr\label{eq:integral-def}
\end{align}
and agree to say that $\int_a^b \varphi(X)\,dY$ in \eqref{eq:integral-def} \emph{exists} if for a certain range of $\alpha$ the integrands on the right-hand side in \eqref{eq:integral-def} are well-defined, the integrals exist and their sum does not depend on the particular choice of $\alpha$ within that range.

Up to details, the following was shown in \cite[proof of Theorem 3.3]{HN09}. 
\begin{theorem}\label{T:HN}
Suppose that $\frac{1}{3}<\beta<\frac{1}{2}$, $(X,Y,X\otimes Y)\in M_{m,d}^\beta([a,b])$ and $\varphi$ is differentiable with $\lambda$-H\"older continuous first order partial derivatives. If
\begin{equation}\label{E:lambdacond}
\frac{1}{\beta}-2<\lambda < 1,\
\end{equation}
then $\int_a^b \varphi(X)\,dY$ in \eqref{eq:integral-def} exists and we have 
\begin{multline}
\Big|\int_a^b\varphi(X)\,dY\Big|\leq c\:\sum_{j=1}^d\Big\lbrace \|\varphi_j(X)\|_{L^1([a,b])}+\lip(\varphi_j)[[X]]_\beta+\sum_{i=1}^m[[\partial_i\varphi_j]]_\lambda[[X]]_\beta^{1+\lambda}\Big\rbrace [[Y^j]]_\beta\notag\\
+c\:\sum_{j=1}^d\sum_{i=1}^m\Big\lbrace \|\partial_i\varphi_j(X)\|_{L^1([a,b])}+[[\partial_i\varphi_j]]_\lambda[[X]]_\beta\Big\rbrace[[(X\otimes Y)_{\cdot,b}^{i,j}]]_{2\beta}.\notag
\end{multline}
\end{theorem}

We state a proof for the convenience of the reader and since parts of this proof will later be recycled to prove our main result Theorem \ref{thm:existence-of-integral-H} below.
\begin{proof}
By \eqref{E:lambdacond} we can find 
\begin{equation}\label{E:alphacond}
1-\beta<\alpha<\frac{\lambda\beta+1}{2}.
\end{equation}
For any $r,\theta\in [a,b]$, let $\gamma^{\theta,r}:[0,1]\to \mathbb{R}^m$ denote the constant speed parametrization 
\begin{equation}\label{E:paramline}
\gamma^{\theta,r}(t) = t X(r) + (1-t)X(\theta)
\end{equation}
of the straight line from $X_\theta$ to $X_r$. For fixed $j$ we have
\begin{equation}\label{E:stokes}
\varphi_j(X)_{\theta,r}=\int_{\gamma^{\theta,r}} d\varphi_j=\int_0^1 D \varphi_j(\gamma^{\theta,r}(t))(\dot{\gamma}^{\theta,r})\:dt=\sum_{i=1}^m \int_0^1 \partial_i \varphi_j(\gamma^{\theta,r}(t))\:dt\:X_{\theta,r}^i
\end{equation}
and consequently
\begin{align}\label{eq:Mbound}
\Big|\varphi_j(X)_{\theta,r}-\sum_{i=1}^m \partial_i\varphi_j(X(\theta))X_{\theta,r}^i\Big|&=\left|\sum_{i=1}^m\int_0^1(\partial_i \varphi_j(\gamma^{\theta,r}(t))-\partial_i \varphi_j(X(\theta)))\:dt\:X_{\theta,r}^i\right|\\
&\leq\:|X_{\theta,r}|\:\sum_{i=1}^m\int_0^1 |\partial_i \varphi_j(\gamma^{\theta,r}(t))-\partial_i \varphi_j(X(\theta))|\:dt\notag\\
&\leq\:\frac{1}{1+\lambda}|X_{\theta,r}|^{1+\lambda}\sum_{i=1}^m[[\partial_i\varphi_j]]_\lambda,\notag
\end{align}
since $|\gamma^{\theta,r}(t)-X(\theta)|\leq t|X(r)-X(\theta)|$ for any $t\in [0,1]$. Using this on the second summand of $\hat{D}^\alpha_{a+} \varphi(X)$ and again \cite[Lemma 4.33]{HTV2020} on the first, we find that 
\begin{align}
\int_a^b \vert  \hat{D}^\alpha_{a+}\varphi_j(X)(r)\vert\: dr&\leq c\:\Big\lbrace \|\varphi_j(X)\|_{W^{\alpha,1}([a,b])}+\sum_{i=1}^m[[\partial_i\varphi_j]]_\lambda[[X]]_\beta^{1+\lambda}\Big\rbrace\notag\\
&\leq c\:\Big\lbrace\|\varphi_j(X)\|_{L^1([a,b])}+\lip(\varphi_j)[[X]]_\beta+\sum_{i=1}^m[[\partial_i\varphi_j]]_\lambda [[X]]_\beta^{1+\lambda}\Big\rbrace.\notag
\end{align}
Here we have used that $\alpha<\beta$ by \eqref{E:alphacond} and since $\beta<\frac12$ and $\lambda<1$, and we have used that $\frac{\lambda \beta-1}{2}+\beta>0$ by \eqref{E:lambdacond} and by \eqref{E:alphacond} therefore $\alpha<\frac{\lambda\beta+1}{2}<(1+\lambda)\beta$.
Similarly as in the proof of Corollary \ref{C:compositioncase} we find that 
\[\int_a^b|D_{a+}^{2\alpha-1}(\partial_i \varphi_j(X))(r)|\:dr\leq c\:\Big\lbrace \|\partial_i\varphi_j(X)\|_{L^1([a,b])}+[[\partial_i\varphi_j]]_\lambda[[X]]_\beta\Big\rbrace,\]
here we use the fact that $2\alpha<\lambda\beta+1$. Since $\beta>1-\alpha$, we have 
\[\|D_{a+}^{1-\alpha}(Y-Y(b))^j\|_{\sup}\leq c\:[[Y^j]]_\beta\quad \text{and}\quad \|D_{b-}^{2-2\alpha}(X\otimes Y)^{i,j}_{\cdot,b}\|_{\sup}\leq c\:[[(X\otimes Y)_{\cdot,b}]]_{2\beta}.\]
This shows the existence of the integral and the stated estimate. The independence of the integral of the particular choice of $\alpha$ is seen similarly as in \cite[Proposition 2.1]{Zahle98}.
\end{proof}

\section{Variability and composition with $BV$-functions}\label{S:varcomp}

\subsection{Existence of integrals in the elementary case} Recall that $\varphi\in L^1_{\loc}(\mathbb{R}^m)$ is of \emph{locally bounded variation} (denoted by $\varphi\in BV_{\loc}(\mathbb{R}^m)$) if its distributional partial derivatives $D_i\varphi$,  $i=1,...,m$, are signed Radon measures. We write $D\varphi=(D_1\varphi,...,D_m\varphi)$ for its $\mathbb{R}^m$-valued gradient measure and $\left\|D\varphi\right\|$ for the total variation measure of $D\varphi$. The measure $\left\|D\varphi\right\|$ is a nonnegative Radon measure. If $\varphi\in L^1(\mathbb{R}^m)$ and $\left\|D\varphi\right\|(\mathbb{R}^m)<\infty$, then $\varphi$ is said to be of \emph{bounded variation}, $\varphi\in BV(\mathbb{R}^m)$.
Details and background can for instance be found in \cite{AFP} or \cite{Ziemer}.

Given $0<\gamma<m$ and a nonnegative Radon measure $\nu$ on $\mathbb{R}^m$, we write 
\begin{equation}\label{E:Rieszpot}
U^\gamma\nu(x)=\int_{\mathbb{R}^m} |x-z|^{\gamma-m}\,\nu(dz),\quad x\in \mathbb{R}^m.
\end{equation}
Up to a constant which depends only on $m$ and $\gamma$, the $[0,\infty]$-valued function $U^\gamma\nu$ is the \emph{Riesz potential of order $\gamma$ of $\nu$}, see \cite{AH96} or \cite{Landkof}.

Given $\varphi \in BV_{\loc}(\mathbb{R}^m)$, $1\leq p<\infty$ and $0<s<1$, we say that a Borel function $X:[a,b]\to\mathbb{R}^m$ is \emph{$(s,p)$-variable w.r.t. $\varphi$} if the function $t\mapsto U^{1-s}\|D\varphi\|(X(t))$ is in $L^p([a,b])$. We write $V(\varphi,s,p)$ for the set of all $X$ integrable over $[a,b]$ and $(s,p)$-variable w.r.t. $\varphi$. See \cite[Definition 2.2]{HTV2020} or \cite[Definition 5.1]{HTV2022}. The following composition results were shown in \cite[Theorem 2.12 (iii)]{HTV2020} and \cite[Theorem 5.11 (i)]{HTV2022}.

\begin{theorem}
\label{thm:sobolev-membership}
Let $X:[a,b]\to\mathbb{R}^m$ be a Borel function, $0<s,\beta<1$ and $1\leq p<\infty$. 
\begin{enumerate}
\item[(i)] If $\varphi\in BV_{\loc}(\mathbb{R}^m)$, $\gamma<s\beta$ and $X \in C^{\beta}([a,b];\mathbb{R}^m)\cap V(\varphi,s,p)$, 
then $\varphi \circ X \in W^{\gamma,p}([a,b])$ and in particular, 
\[[\varphi\circ X]_{\gamma,p}\leq c\: [[X]]_{\beta}^s\: \big\| U^{1-s}\|D\varphi\|(X)\big\|_{L^p([a,b])}\]
with a constant $c>0$ depending only on $a$, $b$, $m$, $s$, $p$, $\beta$ and $\gamma$.
\item[(ii)] If $\varphi\in BV(\mathbb{R}^m)$, $1\leq q <\infty$, $r\geq 1$ is such that $\frac{1}{p}+\frac{s}{q}\leq \frac{1}{r}$, $\gamma<s\beta$ and $X \in W^{\beta,q}([a,b];\mathbb{R}^m)\cap V(\varphi,s,p)$,
then $\varphi \circ X \in W^{\gamma,r}([a,b])$ and in particular, 
\[[\varphi\circ X]_{\gamma,r}\leq c\: [X]_{\beta,q}^s\: \big\|  U^{1-s}\|D\varphi\|(X) \big\|_{L^p([a,b])}\]
with a constant $c>0$ depending only on $a$, $b$, $m$, $s$, $p$, $q$, $r$, $\beta$ and $\gamma$.
\end{enumerate}
\end{theorem}

This gives existence results and estimates for integrals involving $BV$-coefficients in the elementary case, which we quote from \cite[Theorem 2.12 and Corollary 4.36]{HTV2020} and \cite[Corollary 6.4]{HTV2022}. They follow from Proposition \ref{prop:basic-GSL} and Theorem \ref{thm:sobolev-membership} and may be seen as counterparts to Corollary \ref{C:compositioncase} for $BV$-coefficients.

\begin{corollary}\label{C:sobolev-membership}
Suppose that  $0<s,\beta,\delta<1$, $Y\in C^\delta([a,b];\mathbb{R}^d)$ and $X:[a,b]\to\mathbb{R}^m$ is a Borel function.
\begin{enumerate}
\item[(i)] If $\varphi\in BV_{\loc}(\mathbb{R}^m)$, $s\beta+\delta>1$ and $X \in C^{\beta}([a,b];\mathbb{R}^m)\cap V(\varphi,s,1)$, then $\int_a^b\varphi(X)\,dY$ in \eqref{E:basicint} exists and 
\[\Big|\int_a^b\varphi(X)\,dY\Big|\leq c\:\Big\lbrace\|\varphi(X)\|_{L^1([a,b];\mathbb{R}^d)}+[[X]]_\beta^s\| U^{1-s}\|D\varphi\|(X)\|_{L^1([a,b])}\Big\rbrace [[Y]]_\delta.\]
\item[(ii)] If $\varphi\in BV(\mathbb{R}^m)$, $s\beta+\delta>1$, $1\leq p,q<\infty$, $\frac{1}{p}+\frac{s}{q}\leq 1$ and 
$X \in W^{\beta,q}([a,b];\mathbb{R}^m)\cap V(\varphi,s,p)$, then 
$\int_a^b\varphi(X)\,dY$ in \eqref{E:basicint} exists and 
\[\Big|\int_a^b\varphi(X)\,dY\Big|
\leq c\:\Big\lbrace\|\varphi(X)\|_{L^1([a,b];\mathbb{R}^d)}+[X]_{\beta,q}^s\big\|  U^{1-s}\|D\varphi\|(X)
\big\|_{L^p([a,b])}\Big\rbrace [[Y]]_\delta.\]
\end{enumerate}
\end{corollary}

\begin{remark} Variants of Theorem \ref{thm:sobolev-membership} (ii) and Corollary \ref{C:sobolev-membership} (ii) under the less restrictive hypothesis that $Y\in W^{\delta,v}([a,b];\mathbb{R}^d)$ can easily be stated, but come with notationally more involved conditions.
\end{remark}

\subsection{Existence of integrals in the rough case} Our main result is the following, which is a counterpart of Theorem \ref{T:HN} for Lipschitz coefficients with first order partials in $BV$ and a rough counterpart of Corollary \ref{C:sobolev-membership}.

\begin{theorem}
\label{thm:existence-of-integral-H}
Suppose that $\frac{1}{3}<\beta<\frac{1}{2}$, $(X,Y,X\otimes Y)\in M_{m,d}^\beta([a,b])$, $\varphi\in \lip(\mathbb{R}^m;\mathbb{R}^d)$ with $\partial_i\varphi_j\in BV_{\loc}(\mathbb{R}^m)$ for all $i,j$ and 
\begin{equation}\label{E:scond}
\frac{1}{\beta}-2<s < 1.
\end{equation}
If $X$ is $(s,1)$-variable w.r.t. all $\partial_i\varphi_j$ and for some $\frac12(\beta(1+s)-(1-\beta))<\varepsilon<\beta(1+s)-(1-\beta)$ we have 
\begin{equation}\label{E:segment}
\int_a^b\int_a^b|r-\theta|^{\varepsilon-1}\int_0^1U^{1-s}\|D\partial_i\varphi_j\|(\gamma^{\theta,r}(t))t^s\:dt \,d\theta \, dr<\infty
\end{equation}
for all $i$ and $j$, then $\int_a^b \varphi(X)\,dY$ in \eqref{eq:integral-def} exists and we have 
\begin{multline}
\Big|\int_a^b\varphi(X)\,dY\Big|\leq c\:\sum_{j=1}^d\Big\lbrace \|\varphi_j(X)\|_{L^1([a,b])}+\lip(\varphi_j)[[X]]_{\beta}\\
+\sum_{i=1}^m\int_a^b\int_a^b|r-\theta|^{\varepsilon-1}\int_0^1U^{1-s}\|D\partial_i\varphi_j\|( \gamma^{\theta,r}(t))t^s\:dt \,d\theta\, dr\: [[X]]_\beta^{1+s}\label{E:question}\\
+\sum_{i=1}^m \|U^{1-s}\left\|D\partial_i \varphi_j\right\|(X)\|_{L^1([a,b])}\: [[X]]_\beta^{1+s}\Big\rbrace [[Y^j]]_\beta\\
+c\:\sum_{j=1}^d\sum_{i=1}^m\Big\lbrace \|\partial_i\varphi_j(X)\|_{L^1([a,b])}+
\| U^{1-s}\|D\partial_i\varphi_j\|(X)\|_{L^1([a,b])}[[X]]_\beta^s\Big\rbrace[[(X\otimes Y)_{\cdot,b}^{i,j}]]_{2\beta}.
\end{multline}
The segment  $\gamma^{\theta,r}$ in \eqref{E:segment} and \eqref{E:question} is defined as in \eqref{E:paramline}.
\end{theorem}

Note that $1-\beta<\beta(1+s)$ by (\ref{E:scond}). Hypothesis \eqref{E:segment} may be seen as a \emph{weighted variability condition on $X$ w.r.t. the partial derivatives $\partial_i\varphi_j$ along line segments $\gamma^{\theta,r}$}.

We turn to a proof of Theorem \ref{thm:existence-of-integral-H}. Given $0<\gamma<m$, $R>0$ and a nonnegative Radon measure $\nu$ on $\mathbb{R}^m$, we define the \emph{truncated fractional Hardy-Littlewood maximal function of $\nu$ of order $\gamma$} by $\mathcal{M}_{\gamma,R}\nu(x):=\sup_{0<r<R}r^{\gamma-m}\nu(B(x,r))$, $x\in \mathbb{R}^m$. Obviously 
$\mathcal{M}_{\gamma,R}\nu(x)\leq U^{\gamma}\nu(x)$, $x\in\mathbb{R}^m$.

\begin{proof}
We can mimic \eqref{eq:Mbound} and estimate
\begin{equation}\label{E:diff}
\Big|\varphi_j(X)_{\theta,r}-\sum_{i=1}^m \partial_i\varphi_j(X(\theta))X_{\theta,r}\Big|\leq \:|X_{\theta,r}|\:\sum_{i=1}^m\int_0^1 |\partial_i \varphi_j(\gamma^{\theta,r}(t))-\partial_i \varphi_j(X(\theta))|\:dt.
\end{equation}
Using \cite[Proposition C.1]{HTV2020}, which is based on \cite[Lemma 4.1]{AK}, we find that
\begin{align}
&|\partial_i \varphi_j(\gamma^{\theta,r}(t))-\partial_i \varphi_j(X(\theta))|\notag\\
&\leq c|\gamma^{\theta,r}(t)-X(\theta)|^s\big\lbrace \mathcal{M}_{1-s,|\gamma^{\theta,r}(t)-X(\theta)|}\left\|D\partial_i \varphi_j\right\|(\gamma^{\theta,r}(t))\notag\\
&\hspace{180pt} +\mathcal{M}_{1-s,|\gamma^{\theta,r}(t)-X(\theta)|}\left\|D\partial_i \varphi_j\right\|(X(\theta))\big\rbrace\notag\\
&\leq c\:t^s|X_{\theta,r}|^s\big\lbrace U^{1-s}\left\|D\partial_i \varphi_j\right\|(\gamma^{\theta,r}(t))+U^{1-s}\left\|D\partial_i \varphi_j\right\|(X(\theta))\big\rbrace.\notag
\end{align}
Consequently \eqref{E:diff} is bounded by 
\[c\:|X_{\theta,r}|^{1+s}\sum_{i=1}^m\left\lbrace \int_0^1 U^{1-s}\left\|D\partial_i \varphi_j\right\|( \gamma^{\theta,r}(t)) t^s \,dt+\frac{1}{1+s} U^{1-s}\left\|D\partial_i \varphi_j\right\|(X(\theta))\right\rbrace. \]
Now set $\alpha:=\beta(1+s)-\varepsilon$. Then $1-\beta<\alpha<\frac{s\beta+1}{2}<\beta(1+s)$. Similarly as in the proof of Theorem \ref{T:HN}, we therefore have 
\begin{multline}
\int_a^b \vert  \hat{D}^\alpha_{a+}\varphi_j(X)(r)\vert\: dr\leq c\:\Big\lbrace \|\varphi_j(X)\|_{L^1([a,b])}+\lip(\varphi_j)[[X]]_{\beta}\notag\\
+\sum_{i=1}^m\int_a^b\int_a^b|r-\theta|^{\beta(1+s)-\alpha-1}\int_0^1U^{1-s}\|D\partial_i\varphi_j\|( \gamma^{\theta,r}(t))t^s\:dt\, d\theta\, dr\: [[X]]_\beta^{1+s}\notag\\
+\sum_{i=1}^m \|U^{1-s}\left\|D\partial_i \varphi_j\right\|(X)\|_{L^1([a,b])}\: [[X]]_\beta^{1+s}\Big\rbrace
\end{multline}
and
\[\int_a^b|D_{a+}^{2\alpha-1}(\partial_i \varphi_j(X))(r)|\:dr\leq c\:\Big\lbrace \|\partial_i\varphi_j(X)\|_{L^1([a,b])}+
\| U^{1-s}\|D\partial_i\varphi_j\|(X)\|_{L^1([a,b])}[[X]]_\beta^s\Big\rbrace.\]
The remaining arguments are the same as in the proof of Theorem \ref{T:HN}.
\end{proof}

In the one-dimensional case $m=1$ condition (\ref{E:segment}) follows from variability.

\begin{lemma}\label{L:businessasusual}
Let $0<s<1$, $\varepsilon>0$ and $X\in C([a,b],\mathbb{R})$. If $\psi\in BV_{\loc}(\mathbb{R})$ and $X$ is $(s,1)$-variable w.r.t. $\psi$, then $z\mapsto \int_a^b\int_a^b |r-\theta|^{1-\varepsilon}\int_0^1|\gamma^{\theta,r}(t)-z|^{-s}\,dt\, d\theta\, dr$ is in $L^1(\mathbb{R},\|D\psi\|)$ and (\ref{E:segment}) holds with $\psi$ in place of  $\partial_i\varphi_j$.
\end{lemma}

\begin{proof}
Note first that $\sup_{a\leq r\leq b}\int_a^b|r-\theta|^{\varepsilon-1}d\theta\leq \frac{2}{\varepsilon}(b-a)^\varepsilon=:c_{a,b,\varepsilon}$. For any $z\in\mathbb{R}$ we have 
\[\int_a^b\int_a^b |r-\theta|^{1-\varepsilon}\mathbf{1}_{\{X_{\theta,r}=0\}}\int_0^1|\gamma^{\theta,r}(t)-z|^{-s}\,dt\, d\theta\, dr\leq c_{a,b,\varepsilon}\int_a^b|X(r)-z|^{-s}dr.\]
For $r,\theta\in [a,b]$ such that $X_{\theta,r}\neq 0$ let $I^{\theta,r}\subset \mathbb{R}$ be the open interval with endpoints $X(\theta)$ and $X(r)$. If $z\in \mathbb{R}\setminus I^{\theta,r}$, then $|\gamma^{\theta,r}(t)-z|\geq \min(|X(r)-z|, |X(\theta)-z|)$. For any $z\in\mathbb{R}$ we therefore have
\[\int_a^b\int_a^b |r-\theta|^{\varepsilon-1}\mathbf{1}_{\{X_{\theta,r}\neq 0,\: z\notin I^{\theta,r}\}}\int_0^1|\gamma^{\theta,r}(t)-z|^{-s}\,dt\, d\theta\, dr\leq 2c_{a,b,\varepsilon}\int_a^b|X(r)-z|^{-s}dr.\]
For $z\in I^{\theta,r}$, let $t(z)=\frac{z-X(\theta)}{X_{\theta,r}}$. Then $t(z)\in (0,1)$, $z=t(z)X(r)+(1-t(z))X(\theta)$ and
\begin{align}
\int_0^1|\gamma^{\theta,r}(t)-z|^{-s}\,dt&= \int_0^{t(z)}|tX_{\theta,r}+X(\theta)-z|^{-s}\,dt+\int_{t(z)}^1|(1-t)X_{r,\theta}+X(r)-z|^{-s}\,dt\notag\\
&=t(z)\int_0^1(1-\tau)^{-s}\,d\tau\,|X(\theta)-z|^{-s}+(1-t(z))\int_0^1\tau^{-s}\,d\tau\,|X(r)-z|^{-s}\notag\\
&\leq \frac{1}{1-s}(|X(\theta)-z|^{-s}+|X(r)-z|^{-s});\notag
\end{align}
here we have used the change of variables $\tau=\frac{t}{t(z)}$ for the first integral and $\tau=\frac{t-t(z)}{1-t(z)}$ for the second. For any $z\in\mathbb{R}$ we therefore have
\[\int_a^b\int_a^b |r-\theta|^{\varepsilon-1}\mathbf{1}_{\{X_{\theta,r}\neq 0,\: z \in I^{\theta,r}\}}\int_0^1|\gamma^{\theta,r}(t)-z|^{-s}t^s\,dt\, d\theta\, dr\leq \frac{2c_{a,b,\varepsilon}}{1-s}\int_a^b|X(r)-z|^{-s}dr.\]
\end{proof}

\begin{corollary}\label{C:businessasusual}
Let $0<s<1$, $\varepsilon>0$ and $X\in C([a,b],\mathbb{R})$. If $\sup_{z\in\mathbb{R}}\int_a^b|X(r)-z|^{-s}dr<\infty$, then 
\begin{equation}\label{E:unibdsegment}
\sup_{z\in\mathbb{R}}\int_a^b\int_a^b |r-\theta|^{\varepsilon-1}\int_0^1|\gamma^{\theta,r}(t)-z|^{-s}\,dt \,d\theta \,dr<\infty,
\end{equation}
$X \in V(s,1,\psi)$ for any $\psi\in BV(\mathbb{R})$ and also (\ref{E:segment}) holds with any $\psi$ in place of  $\partial_i\varphi_j$.
\end{corollary}

\begin{proof}
The claimed variability of $X$ follow using \cite[Proposition 4.8]{HTV2020}. Lemma \ref{L:businessasusual} gives (\ref{E:unibdsegment}), and (\ref{E:unibdsegment}) gives (\ref{E:segment}) with $\psi$ in place of  $\partial_i\varphi_j$.
\end{proof}

\begin{remark}
\begin{enumerate}
\item[(i)] In the context of rough differential equations driven by fractional Brownian motion $B^H$ with Hurst index $\frac13<H<\frac12$ the regularity of the $\varphi_j$ required for the existence of solutions is differentiability with $\lambda$-H\"older first order partials $\partial_i\varphi_j$, where $\frac{1}{H}-1<1+\lambda$, see \cite[Theorems 10.18 and 10.38]{FrizVictoir}. This is in line with (\ref{E:lambdacond}). Uniqueness holds if at least $\frac{1}{H}=1+\lambda$, see \cite[Theorem 10.43]{FrizVictoir}. In Theorem \ref{thm:existence-of-integral-H} the first order partials $\partial_i\varphi_j$ of the coefficient $\varphi$ may be discontinuous, as long as variability and (\ref{E:segment}) are ensured. 
\item[(ii)] By a careful use of H\"older's inequality similarly as in \cite{HTV2022} one can relax the H\"older assumptions on $X$, $Y$ and $X\otimes Y$ to low degrees of Sobolev regularity; Theorem \ref{thm:existence-of-integral-H} then is a ``limit case''. In view of recent results, \cite{LPT2021}, this might be interesting.
\end{enumerate}
\end{remark}

\section{Applications to Gaussian processes}\label{S:Gauss}

In dimension $m=1$ locally bounded local times and \cite[Propositions 4.14 and Corollary 4.13]{HTV2020} can ensure the hypothesis of Corollary \ref{C:businessasusual}. We give a prototype example. 
 
\begin{example}
If $0\leq a<b<\infty$ and $B^H$ is a one-dimensional fractional Brownian motion over a probability space $(\Omega,\mathcal{F},\mathbb{P})$, then there is an event $\Omega_0\in\mathcal{F}$ with $\mathbb{P}(\Omega_0)=1$ such that for all $\omega\in \Omega_0$, all $0<s<1$ and all $\varepsilon>0$ we have $\sup_{z\in\mathbb{R}}\int_a^b|B^H(r,\omega)-z|^{-s}dr<\infty$. In the case $\frac13<\beta<H\leq \frac12$ there is an event $\Omega_1\in\mathcal{F}$ with $\mathbb{P}(\Omega_1)=1$ such that for all $\omega\in \Omega_1$, $s$ as in (\ref{E:scond}) and $\varepsilon>0$, the realizations $(B^H(\omega),B^H(\omega),(B^H\otimes B^H)(\omega))$ of $(B^H,B^H,(B^H\otimes B^H))$ are in $M_{1,1}^\beta([a,b])$ and for any $\varphi\in \lip(\mathbb{R},\mathbb{R})$ with $\varphi'\in BV(\mathbb{R})$ the path $X:=B^H(\omega)$ is $(s,1)$-variable w.r.t. $\varphi'$ and (\ref{E:segment}) holds. Consequently Theorem \ref{thm:existence-of-integral-H} ensures the existence of $\int_a^b\varphi(B^H(\omega))dB^H(\omega)$  for all $\omega\in \Omega_1$ and all $\varphi$ as specified.
\end{example}

For general $m$, let $\mu$ be a finite nonnegative Borel measure on $\mathbb{R}^m$ and let $A_{m,\mu}$ be the class of all $\psi\in BV(\mathbb{R}^m)$ such that $\|D\psi\|\leq \mu$. Suppose that $(X(t)=(X^1(t),...,X^m(t)))_{t\geq 0}$ is an $m$-dimensional Gaussian process over a probability space $(\Omega,\mathcal{F},\mathbb{P})$ with independent components $X^i$ having a common covariance function $R(r,\theta):=\mathbb{E}[X^1(r)X^1(\theta)]$. We write $N(0,\Sigma)$ for the centered normal distribution on $\mathbb{R}^m$ with covariance matrix $\Sigma$. As in \cite[Lemma 3.1]{Huang2020} we can see that for $\sigma^2>0$ and $Z \sim N(0,\sigma^2I)$ we have 
\begin{equation}\label{E:momentZ}    
    \mathbb{E}|Z-z|^{-m+1-s} \leq c\:\big( \sigma^{-m+1-s} \wedge |z|^{-m+1-s}\big),\quad z\in \mathbb{R}^m,
\end{equation}
with a constant $c>0$ that does not depend on $\sigma^2$ or $z$. Here $I$ is the unit matrix. If 
\begin{equation}\label{E:needed}
\mathbb{E}\int_a^b\int_a^b|r-\theta|^{\varepsilon-1}\int_0^1U^{1-s}\mu(tX(r) + (1-t)X(\theta))\:dt \,d\theta \, dr<\infty,
\end{equation}
then there is an event $\Omega_{s,\varepsilon,\mu}\in \mathcal{F}$ with $\mathbb{P}(\Omega_{s,\varepsilon,\mu})=1$ such that for any $\omega\in \Omega_{s,\varepsilon,\mu}$ and any $\psi\in A_{m,\mu}$ condition (\ref{E:segment}) holds with $\psi$ in place of $\partial_i\varphi_j$.

\begin{lemma}\label{L:Lauri}
Let $0<s<1$, $\varepsilon>0$. Suppose that $R(r,\theta)\geq 0$ for all $r$ and $\theta$ and that 
$C_\mu:=\int_{\mathbb{R}^m}\int_a^b\big(\left[R(\theta,\theta)\right]^{\frac{-m+1-s}{2}} \wedge |z|^{-m+1-s}\big) \,d\theta\,\mu(dz) < \infty$.
Then (\ref{E:needed}) holds.
\end{lemma}

\begin{proof}
By the stated assumptions we have $tX(r) +(1-t)X(\theta) \sim N(0,\sigma_{t,r,\theta}^2 I)$ with 
\[\sigma^2_{t,r,\theta} = \mathbb{E}\left(tX^1(r)+(1-t)X^1(\theta)\right)^2= t^2 R(r,r)+(1-t)^2R(\theta,\theta) + 2t(1-t)R(r,\theta).\]
By (\ref{E:momentZ}) and since $
\sigma_{t,r,\theta}^2 \geq \frac14 R(r,r)
$
for $t\geq \frac12$ and $\sigma_{t,r,\theta}^2 \geq \frac14 R(\theta,\theta)$ for $t< \frac12$, we have
\[\int_a^b\int_a^b |r-\theta|^{\varepsilon-1} \int_0^1 \int_{\mathbb{R}^m} \mathbb{E}|tX(r) + (1-t)X(\theta)  - z|^{-m+1-s}\mu(dz)\,dt \, d\theta\, dr\leq 2c_{a,b,\varepsilon}C_\mu.\]
\end{proof} 

\begin{remark}
A sufficient condition to ensure the existence of $\Omega_{s,\mu}\in\mathcal{F}$ with $\mathbb{P}(\Omega_{s,\mu})=1$ such that $X(\omega)\in V(s,1,\psi)$ for any $\omega\in \Omega_{s,\mu}$ and $\psi\in A_{m,\mu}$ is that 
\begin{equation}\label{E:expect}
\mathbb{E}\int_a^bU^{1-s}\mu(X(\theta))\,d\theta<\infty.
\end{equation}
\end{remark} 

\begin{example}
If $\theta\mapsto R(\theta,\theta)$ is nondecreasing and $\int_a^b R(\theta,\theta)^{\frac{-m+1-s}{2}}d\theta<\infty$,
then by similar arguments as in \cite[Examples 4.23]{HTV2020} we have (\ref{E:expect}), and by Lemma \ref{L:Lauri} also (\ref{E:needed}) holds. Consequently there is an event $\Omega'_{s,\varepsilon,\mu}\in\mathcal{F}$ with $\mathbb{P}(\Omega'_{s,\varepsilon,\mu})=1$ such that for any $\omega\in \Omega'_{s,\varepsilon,\mu}$ and any $\varphi\in \lip(\mathbb{R}^m;\mathbb{R}^d)$ with $\partial_i\varphi_j\in A_{m,\mu}$ for all $i,j$ we have $X(\omega)\in V(s,1,\partial_i\varphi_j)$ and the conditions (\ref{E:segment}) hold for $X(\omega)$. 

This happens in particular if $R(0,0)>0$ or if $a>0$. These cases ensure the applicability of Theorem \ref{thm:existence-of-integral-H} to H\"older continuous stationary Gaussian processes. 

It also happens if $R(0,0)=0$ and $\mathbb{E}(X^1(r)-X^1(\theta))^2=c(|r-\theta|^{2H}+o(|r-\theta|^{2H}))$ with some $0<H<1$. Gaussian processes satisfying these conditions include fractional Brownian motions $B^H$ and basically all fractional type Gaussian processes. In this situation $R(\theta,\theta)\geq c'\theta^{2H}$ for small $\theta>0$, so that $(m-1+s)H<1$ is needed if $a=0$. Since $\frac13<H\leq \frac12$ in Theorem \ref{thm:existence-of-integral-H}, we will then need $m=1,2$ to make sure there is an event of probability one for whose elements $\omega$ the integral
$\int_a^b\varphi(B^H(\omega))dB^H(\omega)$ exists for all $\varphi$ as specified above.
\end{example}

\begin{example}
Suppose that $a=0$, $\mu$ has compact support and $U^{1-s}\mu(0)<+\infty$. By Lemma \ref{L:Lauri} condition (\ref{E:needed}) holds. If $m\geq 3$ and $B^H$ is a fractional Brownian motion $B^H$ with Hurst index $H>\frac13$, then $mH>1$ and $U^{1-s+1/H}\mu(0)<+\infty$, so that by \cite[Examples 4.26 and Corollary 4.25]{HTV2020} the paths of $B^H$ are $\mathbb{P}$-a.s. $(s,1)$-variable w.r.t. all $\psi\in A_{m,\mu}$. If $m\geq 3$ and $\frac13<H\leq \frac12$, then we can find an event of probability one for whose elements $\omega$ the integral
$\int_a^b\varphi(B^H(\omega))dB^H(\omega)$ exists for all $\varphi\in \lip(\mathbb{R}^m;\mathbb{R}^d)$ with all $\partial_i\varphi_j$ in $ A_{m,\mu}$.
\end{example}

\providecommand{\bysame}{\leavevmode\hbox to3em{\hrulefill}\thinspace}
\providecommand{\MR}{\relax\ifhmode\unskip\space\fi MR }
\providecommand{\MRhref}[2]{%
  \href{http://www.ams.org/mathscinet-getitem?mr=#1}{#2}
}
\providecommand{\href}[2]{#2}


\begin{thebibliography}{10}


\bibitem{AK}
D. Aalto and J. Kinnunen, \emph{Maximal functions in {S}obolev spaces},
  Sobolev spaces in mathematics. {I}, Int. Math. Ser. (N. Y.), vol.~8,
  Springer, New York, 2009, pp.~25--67. \MR{2747071}


\bibitem{AH96}
D.~R. Adams and L.~I. Hedberg, \emph{Function {S}paces and {P}otential
  {T}heory}, Grundlehren math. Wiss., vol. 314, Springer-Verlag, Berlin, 1996. \MR{1411441}


\bibitem{AFP}
L. Ambrosio, N. Fusco and D. Pallara, \emph{Functions of bounded
  variation and free discontinuity problems}, The Clarendon Press, Oxford
  University Press, New York, 2000. \MR{1857292}


\bibitem{BesaluNualart2011}
M. Besal\'u and D. Nualart, \emph{Estimates for the solution to
  stochastic differential equations driven by a fractional {B}rownian motion
  with {H}urst parameter $h \in ( \frac13 , \frac12 )$}, Stoch. Dyn.
  \textbf{11} (2011), 243--263. \MR{2836524}


\bibitem{CG:14}
Kh. Chouk and M. Gubinelli, \emph{Rough sheets}, Preprint (2014),
  1--53, \url{https://arxiv.org/abs/1406.7748}.


\bibitem{CQ2002}
L. Coutin and Zh. Qian, \emph{Stochastic analysis, rough path analysis
  and fractional {B}rownian motions}, Probab. Theory Related Fields
  \textbf{122} (2002), no.~1, 108--140. \MR{1883719}


\bibitem{FH:20}
P.~K. Friz and M. Hairer, \emph{A course on rough paths}, Universitext,
  Springer, Cham, 2020, With an introduction to regularity structures. Second
  edition. \MR{4174393}


\bibitem{FrizVictoir}
P.~K. Friz and N.~B. Victoir, \emph{Multidimensional stochastic
  processes as rough paths}, Cambridge Studies in Advanced Mathematics, vol.
  120, Cambridge University Press, Cambridge, 2010, Theory and applications. \MR{2604669}


\bibitem{Gubinelli}
M. Gubinelli, \emph{Controlling rough paths}, J. Funct. Anal.
  \textbf{216} (2004), no.~1, 86--140. \MR{2091358}


\bibitem{HTV2022}
M. Hinz, J.~M. T\"{o}lle and L. Viitasaari, \emph{{S}obolev
  regularity of occupation measures and paths, variability and compositions},
  Electron. J. Probab. \textbf{27} (2022), no.~73, 1--29. \MR{4440066}


\bibitem{HTV2020}
M. Hinz, J.~M. T{\"o}lle and L. Viitasaari, \emph{Variability of
  paths and differential equations with {$BV$}-coefficients}, Ann. Inst. Henri
  Poincar\'e (B) Probab. Stat. \textbf{59} (2023), no.~4, 2036--2082. \MR{4663516}


\bibitem{HN09}
Y. Hu and D. Nualart, \emph{Rough path analysis via fractional
  calculus}, Trans. Amer. Math. Soc. \textbf{361} (2009), no.~5, 2689--2718. \MR{2471936}


\bibitem{Huang2020}
J. Huang, D. Nualart, L. Viitasaari and G. Zheng,
  \emph{Gaussian fluctuations for the stochastic heat equation with colored
  noise}, Stoch. Partial Differ. Equ. Anal. Comput. \textbf{8} (2020), no.~2,
  402--421. \MR{4098872}


\bibitem{Ito2015}
Y.~Ito, \emph{Integrals along rough paths via fractional calculus}, Pot. Anal.
  \textbf{42} (2015), 155--174. \MR{3297991}


\bibitem{Ito2017b}
\bysame, \emph{Extension theorem for rough paths via fractional calculus}, J.
  Math. Soc. Japan \textbf{69} (2017), no.~3, 893--912. \MR{3685030}


\bibitem{Ito2017}
\bysame, \emph{Integration of controlled rough paths via fractional calculus},
  Forum Math. \textbf{29} (2017), no.~5, 1163--1175. \MR{3692031}


\bibitem{Landkof}
N.~S. Landkof, \emph{Foundations of {M}odern {P}otential {T}heory},
  Grundlehren math. Wiss., vol. 180, Springer-Verlag, New York-Heidelberg,
  1972. \MR{350027}


\bibitem{LPT2021}
Ch. Liu, D.~J. Pr\"{o}mel and J. Teichmann, \emph{On {S}obolev rough
  paths}, J. Math. Anal. Appl. \textbf{497} (2021), no.~1, Paper No. 124876,
  21. \MR{4192218}


\bibitem{LyonsQian}
T.~J. Lyons and Zh. Qian, \emph{System control and rough paths}, Oxford
  Mathematical Monographs, Oxford University Press, Oxford, 2002. \MR{2036784}


\bibitem{Lyons1994}
T.~J. Lyons, \emph{Differential equations driven by rough signals. {I}. {A}n
  extension of an inequality of {L}.{C}. {Y}oung}, Math. Res. Lett. \textbf{1}
  (1994), no.~4, 451--464. \MR{1302388}


\bibitem{Lyons98}
\bysame, \emph{Differential equations driven by rough signals}, Rev. Mat.
  Iberoamericana \textbf{14} (1998), no.~2, 215--310. \MR{1654527}


\bibitem{LCL:07}
T.~J. Lyons, M. Caruana, and Th. L\'{e}vy, \emph{Differential
  equations driven by rough paths}, Lecture Notes in Mathematics, vol. 1908,
  Springer, Berlin, 2007. \MR{2314753}

\bibitem{MM23}
T. Matsuda and A. Mayorcas, \emph{Pathwise uniqueness for multiplicative Young and rough
differential equations driven by fractional Brownian motion}, Preprint (2023), \url{https://arxiv.org/abs/2312.06473}.

\bibitem{MP24}
T. Matsuda and N. Perkowski, \emph{An extension of the stochastic sewing lemma and applications
to fractional stochastic calculus}, Forum of Math. Sigma \textbf{12}, (2024), no.~e52, 1--52. \MR{4730255}

\bibitem{NualartRascanu}
D. Nualart and A. R\u{a}\c{s}canu, \emph{Differential equations driven by
  fractional {B}rownian motion}, Collect. Math. \textbf{53} (2002), no.~1,
  55--81. \MR{1893308}


\bibitem{NT:11}
D. Nualart and S. Tindel, \emph{A construction of the rough path above
  fractional {B}rownian motion using {V}olterra's representation}, Ann. Probab.
  \textbf{39} (2011), no.~3, 1061--1096. \MR{2789583}


\bibitem{SKM}
S.~G. Samko, Anatoly~A. Kilbas, and Oleg~I. Marichev, \emph{Fractional
  integrals and derivatives}, Gordon and Breach Science Publishers, Yverdon,
  1993. \MR{1347689}


\bibitem{Young}
L.~C. Young, \emph{An inequality of the {H}\"{o}lder type, connected with
  {S}tieltjes integration}, Acta Math. \textbf{67} (1936), no.~1, 251--282. \MR{1555421}


\bibitem{Zahle98}
M. Z\"{a}hle, \emph{Integration with respect to fractal functions and
  stochastic calculus. {I}}, Probab. Theory Related Fields \textbf{111} (1998),
  no.~3, 333--374. \MR{1640795}


\bibitem{Zahle01}
\bysame, \emph{Integration with respect to fractal functions and stochastic
  calculus. {II}}, Math. Nachr. \textbf{225} (2001), 145--183. \MR{1827093}


\bibitem{Ziemer}
W.~P. Ziemer, \emph{Weakly differentiable functions}, Graduate Texts in
  Mathematics, vol. 120, Springer-Verlag, New York, 1989, Sobolev spaces and
  functions of bounded variation. \MR{1014685}


\end{thebibliography}
\end{document}